\documentclass[11pt]{amsart}
\usepackage{epsfig}
\usepackage{graphics}

\newtheorem{theorem}{Theorem}
\newtheorem{lemma}[theorem]{Lemma}

\theoremstyle{definition}
\newtheorem{definition}[theorem]{Definition}

\newtheorem{remark}[theorem]{Remark}

\theoremstyle{remark}

\begin{document}

\newenvironment{prooff}{\medskip \par \noindent {\it Proof}\ }{\hfill
$\square$ \medskip \par}
    \def\sqr#1#2{{\vcenter{\hrule height.#2pt
        \hbox{\vrule width.#2pt height#1pt \kern#1pt
            \vrule width.#2pt}\hrule height.#2pt}}}
    \def\square{\mathchoice\sqr67\sqr67\sqr{2.1}6\sqr{1.5}6}
\def\pf#1{\medskip \par \noindent {\it #1.}\ }
\def\endpf{\hfill $\square$ \medskip \par}
\def\demo#1{\medskip \par \noindent {\it #1.}\ }
\def\enddemo{\medskip \par}
\def\qed{~\hfill$\square$}

\title[]
{A NOTE ON OVERTWISTED CONTACT STRUCTURES}

\author{EL\.IF YILMAZ}

\address{Department of Mathematics, Columbia University, New York, NY
10027 and Department of Mathematics, Middle East Technical University,
06531 Ankara, Turkey}

\email{eyilmaz@math.columbia.edu and elyilmaz@metu.edu.tr}


\thanks{The author is supported by T\"UB\.ITAK 2214 Grant.}

\date{\today}

\begin{abstract}
In this note, we use the recent work of Honda-Kazez-Mati\'c ~\cite{hkm} to prove that a closed contact $3$-manifold admitting a compatible open book decomposition with a nontrivial monodromy which can be presented as a product of left handed Dehn twists is overtwisted.
\end{abstract}

\maketitle
  \setcounter{secnumdepth}{1}
\setcounter{section}{0}

\section{Introduction}

There is a connection between contact structures and open book decompositions. This connection can be used to classify the contact structures. The monodromy of an open book has a distinguishing role in this classification. By the work of E. Giroux, we know the relation between the Stein fillable contact structures and the monodromy of that supporting open book: For a Stein fillable contact structure, there is at least an open book supporting that contact structure with monodromy product of right handed Dehn twists. Hence we try to understand the relation between mapping classes in the mapping class group of a page and the corresponding contact structures. In this note, we try to give a relation of overtwisted contact structures with the monodromy and find a sufficient condition for overtwistedness by using the \emph{right-veering diffeomorphisms} as studied in the recent work of Honda-Kazez-Mati\'c ~\cite{hkm}. We prove that:

\begin{theorem} \label{cor1} Let $S$ be a compact oriented surface with boundary.
If the monodromy of an open book decomposition $(S,h)$ of a contact $3$-manifold
$ (Y,\xi) $ is a product of left handed Dehn twists about homotopically nontrivial simple closed curves, then the contact structure $ \xi $ supported by the open book
$(S,h)$ is overtwisted.
\end{theorem}

In Section 2, we will give some preliminaries about contact strucures,
open book decompositions and right-veering diffeomorphisms.
In Section 3 we will prove the main theorem.

\textbf{Acknowledgements.}
I would like to thank Mustafa Korkmaz, Andr\'as Stipsicz and \.Inan\c{c} Baykur for helpful discussions.

\section{Preliminaries}
\subsection{Contact Structures and Open Book Decompositions}

A \emph{contact form} on an oriented $3$-manifold $M$ is a one-form $\alpha$ such that $\alpha \wedge d\alpha>0$ everywhere with respect to the orientation of $M$. A \emph{contact structure} on $M$ is a cooriented plane field $\xi$ which is the kernel of some contact form. A \emph{contact manifold} is a manifold equipped with a contact structure.

\par
An open book decomposition of a closed oriented $3$-manifold $M$ is a pair $(S,h)$,
where $S$ is an oriented surface with boundary and $h:S\rightarrow S$ is a diffeomorphism
restricting to the identity on the boundary of $S$. Take $S\times \left[0,1 \right]$ and
identify $(h(x),0)$ with $(x,1)$ for $x\in S$ and $(y,0)$ with $(y,t)$ for $y\in \partial S$
and $t\in \left[0,1 \right]$. Identify this obtained $3$-manifold with $M$;
with this identification $S\times t$, $ t\in \left[0,1 \right]$, is called a \textit{page},
$\partial S $ is called a \textit{binding} and $h$ is the \textit{monodromy} of the open book.

\par A contact structure $\xi$ on a $3$-manifold $M$ is \emph{supported} by an open book of $M$ if $\xi$ is the kernel of a contact form $\alpha$ such that $d\alpha$ is a volume form on every page, the binding is transverse to $\xi$ and the orientation of the binding induced by $\alpha $ agrees with the boundary orientation of the pages.

\subsection{Right-veering Diffeomorphisms}

Let $S$ be a compact connected oriented surface with boundary.
Define the mapping class group of $S$ to be the isotopy classes of (automatically orientation-preserving)
self-diffeomorphisms of the surface $S$ which restrict to the identity
on $ \partial S$ and denote it by $Mod(S,\partial S)$. In ~\cite{hkm}, Honda-Kazez-Mati\'c
introduced the notion of \emph{right-veering diffeomorphisms} and the monoid
$Veer(S,\partial S) \subset Mod(S,\partial S)$ of right-veering diffeomorphisms of $S$.
We recall these notions.

\begin{definition}
Let $\alpha$ and $\beta$ be two properly embedded arcs with a common initial point
$x\in \partial S$. Recall that an arc $\alpha$ on $S$ is properly embedded if $\alpha$
is embedded and $\partial \alpha=\alpha \cap \partial S$. Isotope $\alpha$ and $\beta$
fixing the endpoints so that they intersect transversally with the least possible number
of points and that they are transverse to $ \partial S$. We say that ``$ \beta $ is to
the right of $\alpha$" if $ \alpha =  \beta $ or the tangent vectors $( \beta', \alpha')$ gives the orientation
of $S$ at $x$, and ``$ \beta $ is to the left of $\alpha$" otherwise.

\end{definition}

\begin{definition}
A diffeomorphism $h:S\to S$ is called \emph{right-veering} if for every choice of
basepoint $x \in \partial S$ and every choice of properly embedded arc $\alpha$ based at $x$, $h(\alpha)$ is
to the right of $ \alpha $ at $x$.

It is easy to see that for two isotopic self diffeomorphism $h_1$ and $h_2$ of $S$,
$h_1$ is right veering if and only if $h_2$ is right veering. Therefore,
one can talk about right veering mapping classes. The subset of $Mod(S,\partial S)$ consisting of
right veering elements is denoted by $Veer(S,\partial S)$.
It follows that $Veer(S,\partial S)$ is a monoid.

\end{definition}

In ~\cite{hkm}, it was shown that the monoid $Dehn^+(S,\partial S)\subset Mod(S,\partial S)$
consisting of products of right Dehn twists is a submonoid of  $Veer(S,\partial S)$.
The main result of ~\cite{hkm} is the following theorem.

\begin{theorem}[\cite{hkm}] \label{Th3}
A contact structure  $ (M, \xi) $ is tight
if and only if all of its open book decompositions $(S,h)$ have right-veering monodromy.
\end{theorem}

Throughout this paper, $t_{x}$ denotes the right Dehn twist about the simple closed curve $x$
and $(t_{x})^{-1}$ denotes the left Dehn twist about $x$. We consider the diffeomorphisms of a
surface up to isotopy.

\section{Proof of The Main Theorem}

\begin{lemma}\label{r5}
Let $S$ be a compact connected oriented surface with boundary and $f,h:S\rightarrow S$ be
two orientation preserving diffeomorphisms restricting to the identity on $\partial S$.
If $h$ is right-veering, then $fhf ^ {-1}$ is right-veering.

\end{lemma}

\begin{proof}
Suppose that $ fhf ^ {-1}$ is not right-veering. By definition, there is a point
$x \in \partial S $ and a properly embedded arc $\alpha$ starting at $x$ such that
$ fhf^{-1}(\alpha) $ is to the left of $ \alpha $ at $x$. This implies that, for
an arc $f^{-1}(\alpha) $, $hf^{-1}(\alpha) $ is to the left of $f^{-1}(\alpha)$ at $x $
which contradicts to the assumption that $h$ is a right-veering diffeomorphism.
\end{proof}

\begin{lemma}\label{r6}
Let $c$ be a homotopically nontrivial simple closed curve on an oriented surface $S$ with boundary.
Then the left Dehn twist $t_{c}^{-1} $ is not a right-veering diffeomorphism.

\end{lemma}

\begin{proof}
By the classification of surfaces, there is a diffeomorphism $f$ of $S$ such that $f(c)$ is the curve $a$
in one of pictures shown in Figure~\ref{sep1}. It follows that $t_c^{-1}$ is conjugate to $t_a^{-1}$. Hence,
in order to finish the proof, it suffices to show that $t_a^{-1}$ is not right veering.

In the figure, a properly embedded arc $\alpha$ is presented so that $t_a^{-1}(\alpha)$
is to the left of $\alpha$, which shows that $t_a^{-1}$ is not right veering.
\end{proof}

\begin{figure}[hbt]
   \centering
        \includegraphics[width=10cm]{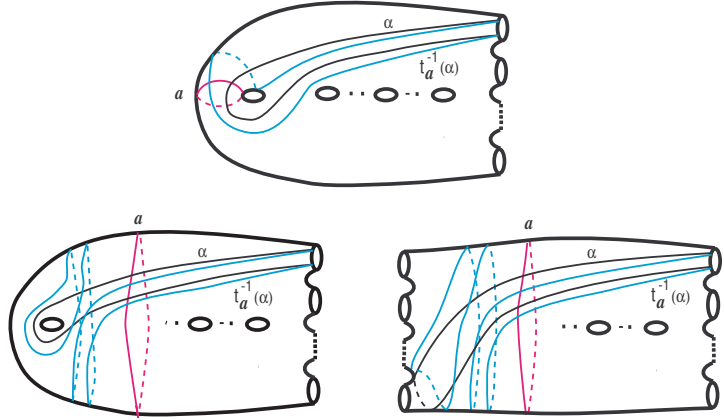}
   \caption{Any homotopically nontrivial simple closed curve is topologically equivalent to one of these curves.}
      \label{sep1}
\end{figure}

\begin{lemma}\label{r7}
Let $a$ be a homotopically nontrivial simple closed curve on an oriented surface $S$ with boundary. If $f \in Veer(S,\partial S) $, then $ t_{a}^{-1}f^{-1} $ is not right-veering.
\end{lemma}

\begin{proof}
Suppose $ t_{a}^{-1}f^{-1} $ is right-veering. $ Veer(S,\partial S) $ is closed under
composition since it is a monoid. Then $ t_{a}^{-1}f^{-1}f $ is a right-veering
diffeomorphism. On the other hand, $ t_{a}^{-1} = t_{a}^{-1}f^{-1}f $ and in Lemma \ref{r6}
we showed that it is not right-veering which contradicts to the assumption.
\end{proof}

\subsection{Proof of Theorem \ref{cor1}}

Let the monodromy of the open book $(S,h)$ of $(Y,\xi)$ be
$h=t_{c_{1}}^{-1}t_{c_{2}}^{-1} \cdots t_{c_{n}}^{-1}$ such that each $c_i$ is homotopically nontrivial, where $ 1 \leq i \leq n $. By Lemma \ref{r6},  $t_{c_{1}}^{-1} $ is not a right-veering diffeomorphism and $ (t_{c_{2}}^{-1} \cdots t_{c_{n-1}}^{-1} t_{c_{n}}^{-1})^{-1} = t_{c_{n}}t_{c_{n-1}} \cdots t_{c_{2}} \in Veer(S,\partial S) $. Then by Lemma \ref{r7}, $h$ is not
right veering. Hence the contact structure $ \xi $ which is supported by the open book $(S,h)$ is overtwisted by Theorem \ref{Th3}. \ \ \ \ \ \ \ \ \ \ \ \ \ \ \ \ \ \ \ \ \ \ \ \ \ \ \ \ \ \ \ \ \ \ \ \ \ \ \ \ \ \ \ \ \ \ \ \ \ \ \ \ \ \ \ $\square$

\begin{remark}
We firstly tried to prove Theorem \ref{cor1}, by using the \emph {sobering arc} argument introduced
by Goodman in ~\cite{g}. In order to find a sobering arc, we must find an arc which has always negative intersections in the interior. However this is not the case for the diffeomorphisms consisting of Dehn twists about the separating curves.
\end{remark}

\begin{remark}

The recent article of  Harvey-Kawamuro-Plamenevskaya ~\cite{pl}
appeared at the arxiv contains a result similar to ours.
Their result states
that if the monodromy is presented as product of left Dehn twists about the standard
generators (nonseparating curves), then the corresponding contact structure is overtwisted. Since our proof also covers the cases where some of the curves are separating, we decided to post it.

\end{remark}

\end{document}